\documentclass[11pt]{article}

\usepackage[a4paper,margin=30mm]{geometry}
\usepackage[T1]{fontenc}
\usepackage{lmodern}
\usepackage{microtype}
\usepackage{amsmath,amssymb,amsthm,mathtools}
\usepackage{booktabs}
\usepackage{enumitem}
\usepackage[hidelinks]{hyperref}
\usepackage{xcolor}

\definecolor{heading}{RGB}{35,55,78}
\usepackage{sectsty}
\sectionfont{\color{heading}}
\subsectionfont{\color{heading}}

\newtheorem{theorem}{Theorem}[section]
\newtheorem{lemma}[theorem]{Lemma}
\newtheorem{proposition}[theorem]{Proposition}
\newtheorem{corollary}[theorem]{Corollary}
\theoremstyle{remark}
\newtheorem{remark}[theorem]{Remark}

\newcommand{\C}{\mathbb C}

\newcommand{\T}{\mathbb T}
\newcommand{\ZF}{Z_F}
\newcommand{\ZFN}{Z_{F,N}}
\newcommand{\Ree}{\operatorname{Re}}
\newcommand{\JF}{J_F}
\newcommand{\ZenodoDOI}{10.5281/zenodo.22300938}

\title{\bfseries The Right Edge of the Zero Set of the Fibonacci Zeta Function}
\author{Marco Mantovanelli\\[2pt]
\small Independent Researcher, Germany\\[1pt]
\small \href{mailto:marco@mantovanelli.de}{\texttt{marco@mantovanelli.de}}\\[1pt]
\small \href{https://orcid.org/0009-0002-0631-293X}{ORCID: 0009-0002-0631-293X}}
\date{}

\begin{document}
\maketitle

\begin{abstract}
Let $F_1=F_2=1$, $F_{n+2}=F_{n+1}+F_n$, and
\[
\ZF(s)=\sum_{n\ge 1}F_n^{-s},\qquad \Ree s>0.
\]
We determine the exact right edge of the closure of the real parts of zeros in the half-plane of absolute convergence.  If $\sigma_F$ is the unique solution of
\[
\ZF(\sigma_F)=4+2\,144^{-\sigma_F},
\]
then
\[
\sigma_F=0.743163398726901648\ldots,
\qquad
\ZF(s)\neq0\quad(\Ree s\ge\sigma_F),
\]
and
\[
\overline{\{\Ree\rho:\ZF(\rho)=0,\ \Ree\rho>0\}}=[0,\sigma_F].
\]
The edge is sharp in a strong almost-periodic sense: zeros occur with relatively dense ordinates near every admissible vertical line.  We quantify their approach to the edge by proving phase locking of a growing set of Fibonacci terms and a Diophantine zero-free cusp; in particular the least height $T_F(\varepsilon)$ of a zero with $\sigma_F-\varepsilon<\Ree\rho<\sigma_F$ satisfies
\[
T_F(\varepsilon)\gg_\eta \varepsilon^{-1/(15.232+\eta)}.
\]
We also determine the Jessen function of the series structurally: it is $C^\infty$, strictly convex on $(0,\sigma_F)$, equals $\log2$ for $\sigma\ge\sigma_F$, and its Riesz measure gives a smooth mean vertical zero density that is supported up to the edge and vanishes there to infinite order.  For every partial sum with $N\ge12$ we obtain the exact right edge $\sigma_N$ of the closure of its positive zero real parts, prove $\sigma_N\nearrow\sigma_F$, and give an explicit exponential asymptotic for $\sigma_F-\sigma_N$.  The same phase mechanism yields a finite-core theorem for positive integral Lucas zeta functions; as an example, the corresponding closure edge for the Pell zeta function is $0.875850904477574106\ldots$.  Finally, using the known meromorphic continuation, we construct a natural $q$-Pochhammer pole-removing completion that is entire of exact order $2$ and type $\log\varphi/4$, yielding a quadratic global bound for the number of zeros.  The common arithmetic source of these results is the exceptional multiplicative dependence at $F_6=2^3$ and $F_{12}=2^4 3^2$, together with primitive-divisor theorems beyond this finite core.
\end{abstract}



\section{Introduction}

The Fibonacci zeta function
\[
\ZF(s)=\sum_{n\ge1}F_n^{-s}
\]
converges absolutely for $\Ree s>0$, since $F_n\asymp \varphi^n$, where $\varphi=(1+\sqrt5)/2$.  Its meromorphic continuation was obtained independently by Egami and Navas; see \cite{Egami1999,Navas2001}.  Further accounts and recent generalizations appear in \cite{Murty2013,AssafKuanLowryDudaWalker2024}.  Mora studied the zeros of the partial sums and, through the corresponding Henry upper bounds, obtained the zero-free half-plane
\[
\ZF(s)\neq0\qquad (\Ree s>\eta),
\]
where $\eta$ is the unique positive solution of $\ZF(\eta)=4$ and
\[
\eta=0.757054949690654898\ldots;
\]
see \cite{Mora2023} and the discussion in \cite{MoraNavasVarona2025}.  The problem addressed here is to replace this upper bound by the actual right boundary of the closure of the zero real parts and then to understand what the same arithmetic mechanism says about the distribution, height, finite truncations, related Lucas zeta functions, and the global meromorphic continuation.

The key point is that the phases of the Fibonacci terms are not independent.  The first relevant dependencies are
\[
F_3=2,\qquad F_4=3,\qquad F_6=8=2^3,\qquad F_{12}=144=2^4 3^2.
\]
The indices $6$ and $12$ are precisely the exceptional indices, apart from $1,2$, in Carmichael's primitive-divisor theorem for Fibonacci numbers; see \cite{Yabuta2001}.  Once this finite dependent block is isolated, the remaining terms acquire new prime phase freedom.  This leads to a finite-dimensional torus extremal problem for the zero-free boundary and, in the opposite direction, to a phase-realization argument that makes the boundary sharp.

Define
\begin{equation}\label{eq:sigma-def}
H(\sigma):=\ZF(\sigma)-4-2\,144^{-\sigma}.
\end{equation}
Our main result is the following.

\begin{theorem}\label{thm:main}
There is a unique $\sigma_F\in(0.7,0.75)$ with $H(\sigma_F)=0$.  It satisfies
\[
\sigma_F=0.743163398726901648\ldots.
\]
Moreover,
\begin{enumerate}[label=\textup{(\roman*)}]
\item $\ZF(s)\neq0$ for $\Ree s\ge\sigma_F$;
\item if
\[
\mathcal R_F:=\{\Ree\rho:\ZF(\rho)=0,\ \Ree\rho>0\},
\]
then
\[
\overline{\mathcal R_F}=[0,\sigma_F];
\]
\item for $0<\varepsilon<\sigma_F-1/2$, put
\[
T_F(\varepsilon):=\min\bigl\{|\Im\rho|:\ZF(\rho)=0,\ \sigma_F-\varepsilon<\Ree\rho<\sigma_F\bigr\}.
\]
Then, for every $\eta>0$,
\begin{equation}\label{eq:T-lower-main}
T_F(\varepsilon)\gg_\eta \varepsilon^{-1/(15.232+\eta)}
\qquad(\varepsilon\downarrow0).
\end{equation}
\end{enumerate}
In particular,
\[
\sup\mathcal R_F=\sigma_F,
\]
but no zero lies on the line $\Ree s=\sigma_F$.
\end{theorem}

Theorem~\ref{thm:main} is the organizing result, but several later statements are intended to show that the edge is part of a broader zero geometry rather than an isolated constant.  First, uniform almost periodicity implies that zeros with real part near any prescribed $\sigma\in(0,\sigma_F]$ occur with relatively dense ordinates.  A quantitative stability form of the torus bound then forces increasingly many Fibonacci phases into the extremal configuration as $\Ree\rho\uparrow\sigma_F$; together with Rhin's estimate for the irrationality measure of $\log3/\log2$ \cite{Rhin1987}, this gives the cusp estimate in part~(iii).

Second, the same almost-periodic model admits a Jessen function
\[
\JF(\sigma)=\lim_{T\to\infty}\frac1{2T}\int_{-T}^{T}\log|\ZF(\sigma+it)|\,dt.
\]
Using the Jessen--Tornehave theory \cite{Jessen1933,JessenTornehave1945}, together with the independent phase freedom supplied by primitive divisors, we prove that $\JF$ is smooth, strictly convex in the zero-bearing region, and exactly constant for $\sigma\ge\sigma_F$.  Its Riesz measure therefore gives a canonical mean vertical zero density; the resulting density reaches every point of the admissible interval and becomes flat to infinite order at the right edge.

Third, the finite-core mechanism persists in two different directions.  For every partial sum $\ZFN(s)=\sum_{n\le N}F_n^{-s}$ with $N\ge12$, the exact right edge of the closure of its positive zero real parts is again controlled by the same $F_{12}=144$ correction, and the edges $\sigma_N$ converge exponentially to $\sigma_F$.  For arithmetic Lucas sequences, the Bilu--Hanrot--Voutier primitive-divisor theorem makes the principle genuinely finite: after the first $30$ terms, new primitive primes provide the required phase freedom.  This yields an exact finite-torus characterization of the closure of the positive real parts of zeros.  Lucas zeta functions and their meromorphic continuation were studied by Kamano \cite{Kamano2013}, while Pan and Liu recently obtained zero-free regions for this wider family \cite{PanLiu2026}; our finite-core result addresses the complementary sharpness problem.  The Pell zeta function provides a simple fully explicit example.

Finally, the known meromorphic continuation allows the local zero geometry in $\Ree s>0$ to be connected with the global analytic structure.  A $q$-Pochhammer factor removes the complete pole lattice and produces an entire completion of exact order $2$ and type $\log\varphi/4$.  This gives a canonical Hadamard product over the zeros and a quadratic global zero bound.  Thus the paper has two complementary scales: a sharp arithmetic phase analysis in the half-plane of convergence and an entire-function completion controlling the global zero divisor.

For orientation, Sections~2--4 establish the sharp boundary and its almost-periodic sharpness.  Section~\ref{sec:lucas} develops the finite-core Lucas principle and the Pell example.  Section~\ref{sec:jessen} treats the Jessen function and mean zero density, and Section~\ref{sec:height} quantifies the approach to the right edge.  The following section determines the exact edges of the partial sums, Section~\ref{sec:completion} constructs the global pole-removing completion, and the final section records numerical values and reproducibility checks.

\section{The exceptional four-term block}

For $\sigma>0$ put
\[
A=2^{-\sigma},\qquad B=3^{-\sigma},\qquad C=8^{-\sigma},\qquad D=144^{-\sigma},
\]
and define
\begin{equation}\label{eq:L-R}
L(\sigma)=2-A-B-C+D,
\qquad
R(\sigma)=\sum_{\substack{n\ge5\\ n\neq6,12}}F_n^{-\sigma}.
\end{equation}
Then
\begin{equation}\label{eq:H-RL}
H(\sigma)=R(\sigma)-L(\sigma).
\end{equation}

\begin{lemma}\label{lem:H}
The function $H$ is positive on $(0,\sigma_F)$, vanishes at $\sigma_F$, and is negative on $(\sigma_F,\infty)$.  Moreover $\sigma_F\in(0.7,0.75)$ is unique.
\end{lemma}

\begin{proof}
For every $\sigma>0$,
\[
H'(\sigma)
=-\sum_{\substack{n\ge3\\ n\neq12}}(\log F_n)F_n^{-\sigma}
 +(\log144)144^{-\sigma}.
\]
The four terms $n=3,4,5,6$ already dominate the positive term.  Indeed, after multiplication by $144^\sigma$ their contribution is
\[
 (\log2)72^\sigma+(\log3)48^\sigma
 +(\log5)(144/5)^\sigma+(\log8)18^\sigma
 >\log(2\cdot3\cdot5\cdot8)=\log240>\log144.
\]
Hence $H'(\sigma)<0$ on $(0,\infty)$.  Moreover $H(\sigma)\to+\infty$ as $\sigma\downarrow0$, since $\ZF(\sigma)\to+\infty$.  Direct summation gives
\[
H(0.7)=0.1557443681\ldots>0,
\qquad
H(0.75)=-0.0230711816\ldots<0.
\]
For example, these signs are certified by summing through $F_{80}$ and using
\[
\sum_{n>N}F_n^{-\sigma}
\le \frac{\varphi^{-(N-1)\sigma}}{1-\varphi^{-\sigma}},
\qquad F_n\ge\varphi^{n-2}.
\]
The claim follows.
\end{proof}

The next lemma is the key improvement over the ordinary triangle inequality.

\begin{lemma}[Torus lower bound]\label{lem:torus}
For every $\sigma\ge1/2$ and $z,w\in\T$,
\begin{equation}\label{eq:torus-bound}
\left|2+Az+Bw+Cz^3+Dz^4w^2\right|\ge L(\sigma).
\end{equation}
Moreover equality can hold only when $z=w=-1$.
\end{lemma}

\begin{proof}
Write $z=-e^{iu}$ and $w=-e^{iv}$.  Set
\[
P=2+Az+Bw+Cz^3+Dz^4w^2.
\]
Then
\begin{align*}
\Ree P-L(\sigma)
&=A(1-\cos u)+B(1-\cos v)+C(1-\cos3u)\\
&\qquad-D(1-\cos(4u+2v)).
\end{align*}
By the triangle inequality,
\[
|e^{i(4u+2v)}-1|
\le |e^{iu}-1|+|e^{3iu}-1|+2|e^{iv}-1|.
\]
Weighted Cauchy--Schwarz therefore yields
\begin{align*}
D|e^{i(4u+2v)}-1|^2
&\le K_\sigma\bigl(
A|e^{iu}-1|^2+C|e^{3iu}-1|^2+B|e^{iv}-1|^2
\bigr),
\end{align*}
where
\[
K_\sigma=72^{-\sigma}+18^{-\sigma}+4\,48^{-\sigma}
\le \frac1{\sqrt{72}}+\frac1{\sqrt{18}}+\frac4{\sqrt{48}}<1.
\]
Using $|e^{ix}-1|^2=2(1-\cos x)$ gives $\Ree P\ge L(\sigma)$.  Since
\[
L(\sigma)\ge2-2^{-1/2}-3^{-1/2}-8^{-1/2}>0,
\]
we obtain $|P|\ge\Ree P\ge L(\sigma)$.

Because $K_\sigma<1$, equality forces
\[
|e^{iu}-1|=|e^{iv}-1|=0,
\]
so $z=w=-1$.
\end{proof}

\section{The sharp zero-free half-plane}

\begin{proposition}\label{prop:zerofree}
The Fibonacci zeta function has no zero in $\Ree s\ge\sigma_F$.
\end{proposition}

\begin{proof}
Let $s=\sigma+it$ and suppose $\sigma\ge\sigma_F$.  Put
\[
z=2^{-it},\qquad w=3^{-it}.
\]
Using $8^{-it}=z^3$ and $144^{-it}=z^4w^2$, write
\[
\ZF(s)=P_\sigma(z,w)+E(s),
\]
where
\[
P_\sigma(z,w)=2+Az+Bw+Cz^3+Dz^4w^2
\]
and
\[
|E(s)|\le R(\sigma).
\]
If $\ZF(s)=0$, Lemma~\ref{lem:torus} gives
\[
L(\sigma)\le |P_\sigma(z,w)|=|E(s)|\le R(\sigma),
\]
so $H(\sigma)=R(\sigma)-L(\sigma)\ge0$.  Lemma~\ref{lem:H} implies $\sigma\le\sigma_F$; hence necessarily $\sigma=\sigma_F$ and equality holds throughout.

The equality statement in Lemma~\ref{lem:torus} then forces
\[
2^{-it}=3^{-it}=-1.
\]
Thus for some integers $k,\ell$,
\[
t\log2=(2k+1)\pi,
\qquad
t\log3=(2\ell+1)\pi,
\]
which would make $\log2/\log3$ rational.  This is impossible, since a rational relation would imply $2^a=3^b$ for nonzero integers $a,b$.  Hence no zero occurs on the boundary either.
\end{proof}

\section{Sharpness and density of real parts}

We now show that the boundary from Proposition~\ref{prop:zerofree} is sharp.  The arithmetic input is the primitive-divisor theorem for Fibonacci numbers.

\begin{lemma}[Phase realization]\label{lem:phase}
Let $u_n\in\T$ be prescribed for $n\ge3$, with the only compatibility conditions
\[
u_6=u_3^3,
\qquad
u_{12}=u_3^4u_4^2.
\]
Then there exists a completely multiplicative function $\chi:\mathbb N\to\T$ such that
\[
\chi(F_n)=u_n\qquad(n\ge3).
\]
\end{lemma}

\begin{proof}
Carmichael's theorem states that every $F_n$ with $n\neq1,2,6,12$ has a primitive prime divisor; see \cite{Yabuta2001}.  Proceed inductively in $n$.  At a nonexceptional index choose a primitive prime $q_n\mid F_n$, and write $e_n=v_{q_n}(F_n)$.  Its phase has not been assigned at any earlier stage, so an appropriate $e_n$th root on the unit circle can be chosen for $\chi(q_n)$ to force $\chi(F_n)=u_n$; any other newly appearing prime may be assigned phase $1$.  At $n=6,12$ the displayed compatibility conditions are exactly those imposed by $F_6=2^3$ and $F_{12}=2^4 3^2$.  Finally assign phase $1$ to primes that never entered the construction.
\end{proof}

We also use the standard polygon fact that if positive lengths $r_j$ satisfy $\sum r_j=R<\infty$ and $\max r_j\le R/2$, then the set of all convergent sums $\sum r_je^{i\theta_j}$ is the closed disk $|z|\le R$.  This follows first for finite polygons and then by compactness.

\begin{lemma}\label{lem:twistzero}
For every $0<\sigma\le\sigma_F$ there exists a completely multiplicative $\chi_\sigma:\mathbb N\to\T$ such that
\[
\ZF(s;\chi_\sigma):=\sum_{n\ge1}\chi_\sigma(F_n)F_n^{-s}
\]
satisfies
\[
\ZF(\sigma;\chi_\sigma)=0.
\]
\end{lemma}

\begin{proof}
Fix
\[
u_3=u_4=-1,
\]
so that the forced phases are $u_6=-1$ and $u_{12}=+1$.  The contribution of the four dependent terms together with the two initial $1$'s is then $L(\sigma)$.

The remaining lengths have total mass $R(\sigma)$ and largest member $5^{-\sigma}$.  Since $\sigma_F<3/4$,
\begin{align*}
\frac{R(\sigma)}{5^{-\sigma}}
&\ge 1+\left(\frac5{13}\right)^\sigma
       +\left(\frac5{21}\right)^\sigma
       +\left(\frac5{34}\right)^\sigma\\
&\ge 1+\left(\frac5{13}\right)^{3/4}
       +\left(\frac5{21}\right)^{3/4}
       +\left(\frac5{34}\right)^{3/4}
>2.066.
\end{align*}
Thus the polygon sums of the free terms fill the disk $|z|\le R(\sigma)$.

It remains to check $|L(\sigma)|\le R(\sigma)$.  If $L(\sigma)\ge0$, this follows from $H(\sigma)=R(\sigma)-L(\sigma)\ge0$ for $\sigma\le\sigma_F$.  If $L(\sigma)<0$, then
\[
R(\sigma)+L(\sigma)>0,
\]
because $R$ contains the term $5^{-\sigma}$, while $2^{-\sigma}+3^{-\sigma}<2$ and $8^{-\sigma}<5^{-\sigma}$.  Hence again $|L(\sigma)|<R(\sigma)$.

Choose phases for the free terms so that their sum is $-L(\sigma)$.  Lemma~\ref{lem:phase} realizes these phases by a completely multiplicative $\chi_\sigma$.
\end{proof}

\begin{lemma}[Relatively dense vertical approximation]\label{lem:vertical}
Let $\chi:\mathbb N\to\T$ be completely multiplicative.  For every compact set
$K\subset\{\Ree s>0\}$ and every $\varepsilon>0$, the set
\[
\mathcal E(K,\varepsilon;\chi)
:=
\left\{\tau\in\mathbb R:
\sup_{s\in K}|\ZF(s+i\tau)-\ZF(s;\chi)|<\varepsilon
\right\}
\]
is relatively dense in $\mathbb R$.  In particular there is a sequence $t_j\to\infty$ such that
\[
\ZF(s+it_j)\longrightarrow \ZF(s;\chi)
\]
locally uniformly in $\Ree s>0$.
\end{lemma}

\begin{proof}
Fix $K$ and $\varepsilon$.  Absolute and locally uniform convergence allows us to choose $M$ so large that the tails of both $\ZF(s+i\tau)$ and $\ZF(s;\chi)$ after $F_M$ have total modulus less than $\varepsilon/3$, uniformly for $s\in K$ and $\tau\in\mathbb R$.  Let $P$ be the finite set of primes dividing $F_1\cdots F_M$.

The numbers $\{\log p:p\in P\}$ are linearly independent over $\mathbb Q$.  Hence the Kronecker flow
\[
\tau\longmapsto (p^{-i\tau})_{p\in P}
\]
is dense in the torus $\T^P$.  By compactness of the torus, the hitting times of any fixed neighborhood of the target point $(\chi(p))_{p\in P}$ are relatively dense.  Choosing that neighborhood sufficiently small makes the first $M$ terms differ by less than $\varepsilon/3$ uniformly on $K$.  Together with the two tail estimates this proves the claim.
\end{proof}

\begin{proof}[Proof of Theorem~\ref{thm:main}\textup{(ii)}]
Fix $0<\sigma\le\sigma_F$.  By Lemma~\ref{lem:twistzero}, choose $\chi_\sigma$ with $\ZF(\sigma;\chi_\sigma)=0$.  By Lemma~\ref{lem:vertical}, vertical translates of $\ZF$ converge locally uniformly to $\ZF(\,\cdot\,;\chi_\sigma)$.  The latter is not identically zero, since it tends to $2$ as $\Ree s\to+\infty$.  Hurwitz's theorem therefore gives zeros of the translates arbitrarily close to $\sigma$.  Equivalently, $\ZF$ has zeros whose real parts tend to $\sigma$.

Thus $(0,\sigma_F]\subset\overline{\mathcal R_F}$, and hence $[0,\sigma_F]\subset\overline{\mathcal R_F}$.  Proposition~\ref{prop:zerofree} gives the reverse inclusion.
\end{proof}

The same approximation mechanism gives substantially more than density of the real parts.

\begin{theorem}[Syndetic ordinates near every admissible line]\label{thm:syndetic}
Fix $0<\sigma\le\sigma_F$ and $\varepsilon>0$.  Then there is $L=L(\sigma,\varepsilon)>0$ such that every interval of length $L$ on the imaginary axis contains the ordinate of a zero $\rho$ of $\ZF$ satisfying
\[
|\Ree\rho-\sigma|<\varepsilon.
\]
For $\sigma=\sigma_F$ these zeros necessarily satisfy
\[
\sigma_F-\varepsilon<\Ree\rho<\sigma_F.
\]
Consequently, if
\[
N_{\sigma,\varepsilon}(T)
=\#\{\rho:\ZF(\rho)=0,\ 0<\Im\rho\le T,\ |\Ree\rho-\sigma|<\varepsilon\},
\]
with zeros counted with multiplicity, then
\[
N_{\sigma,\varepsilon}(T)\gg_{\sigma,\varepsilon} T.
\]
\end{theorem}

\begin{proof}
Choose $\chi_\sigma$ as in Lemma~\ref{lem:twistzero}.  Since $\ZF(s;\chi_\sigma)$ is not identically zero, its zero at $s=\sigma$ is isolated.  Choose
\[
0<r<\min(\varepsilon,\sigma/2)
\]
so that the circle $|s-\sigma|=r$ contains no zero of the twisted function, and put
\[
m=\min_{|s-\sigma|=r}|\ZF(s;\chi_\sigma)|>0.
\]
Lemma~\ref{lem:vertical}, applied to this circle with accuracy $m$, gives a relatively dense set of shifts $\tau$ for which Rouch\'e's theorem guarantees that $\ZF(s+i\tau)$ has a zero in $|s-\sigma|<r$.  Equivalently, $\ZF$ has a zero within distance $r$ of $\sigma+i\tau$.

To pass from relatively dense approximation shifts to relatively dense zero ordinates, let $L_0$ be such that every real interval of length $L_0$ meets the set of admissible shifts.  Given an interval $[a,a+L_0+2r]$, choose
\[
\tau\in[a+r,a+r+L_0]
\]
from that set.  The zero supplied by Rouch\'e has ordinate in $[\tau-r,\tau+r]\subset[a,a+L_0+2r]$.  Thus the zero ordinates are relatively dense, with syndetic length at most $L_0+2r$.  When $\sigma=\sigma_F$, Proposition~\ref{prop:zerofree} forces every one of them to lie strictly to the left of the boundary.  Finally, choosing one zero from a sequence of disjoint intervals of a fixed sufficiently large length gives the stated linear lower bound.
\end{proof}

\section{A finite-core principle for Lucas zeta functions}\label{sec:lucas}

The phase-realization argument above is not peculiar to the Fibonacci sequence.  In the arithmetic Lucas setting, the theorem of Bilu, Hanrot and Voutier makes the same mechanism finite dimensional.  We record the resulting reduction because it separates two issues cleanly: a finite multiplicative-dependence core and an infinite tail whose phases are free.

Let $P,Q\in\mathbb Z$ be coprime, let $P^2-4Q>0$, and let
\[
U_0=0,\qquad U_1=1,\qquad U_{n+2}=P U_{n+1}-Q U_n
\]
be a nondegenerate Lucas sequence with $U_n>0$ for $n\ge1$ and exponential growth.  Its Lucas zeta function is
\[
Z_U(s)=\sum_{n\ge1}U_n^{-s},\qquad \Ree s>0.
\]
Fix an integer $N$ such that every $U_n$ with $n>N$ has a primitive prime divisor, that is, a prime divisor which divides no earlier $U_j$.  For the Lucas sequences under consideration one may always take $N=30$ by the theorem of Bilu--Hanrot--Voutier \cite{BHV2001}.

Let $\mathcal P_N$ be the finite set of primes dividing $U_1\cdots U_N$.  For $\omega=(\omega_p)_{p\in\mathcal P_N}\in\T^{\mathcal P_N}$ put
\begin{equation}\label{eq:lucas-core}
C_N(\sigma;\omega)
=\sum_{n=1}^{N}U_n^{-\sigma}
 \prod_{p\in\mathcal P_N}\omega_p^{v_p(U_n)},
\end{equation}
and define
\[
m_N(\sigma)=\min_{\omega}|C_N(\sigma;\omega)|,
\qquad
M_N(\sigma)=\max_{\omega}|C_N(\sigma;\omega)|.
\]
For the tail set
\[
R_N(\sigma)=\sum_{n>N}U_n^{-\sigma},
\qquad
A_N(\sigma)=\max_{n>N}U_n^{-\sigma},
\qquad
\ell_N(\sigma)=\max\{2A_N(\sigma)-R_N(\sigma),0\}.
\]
The quantity $\ell_N$ is the inner radius in the polygon problem for the tail lengths.

\begin{theorem}[Finite-core criterion]\label{thm:lucas-core}
Let
\[
\mathcal R_U=\{\Ree\rho:Z_U(\rho)=0,\ \Ree\rho>0\}.
\]
Then
\begin{equation}\label{eq:lucas-core-criterion}
\overline{\mathcal R_U}\cap(0,\infty)
=
\left\{\sigma>0:
[m_N(\sigma),M_N(\sigma)]
\cap
[\ell_N(\sigma),R_N(\sigma)]\ne\varnothing
\right\}.
\end{equation}
Equivalently, the condition is
\[
m_N(\sigma)\le R_N(\sigma)
\qquad\text{and}\qquad
\ell_N(\sigma)\le M_N(\sigma).
\]
In particular, for every positive nondegenerate coprime integral Lucas sequence, the closure of the real parts of the zeros in the half-plane of absolute convergence is determined by a finite torus involving only $U_1,\ldots,U_{30}$ and the one-dimensional tail data above.
\end{theorem}

\begin{proof}
The set of possible core moduli $|C_N(\sigma;\omega)|$ is the interval $[m_N(\sigma),M_N(\sigma)]$: the torus is compact and connected, and the modulus is continuous.

For the tail lengths $r_n=U_n^{-\sigma}$, $n>N$, the standard polygon theorem, followed by a compactness argument, shows that the set of all convergent sums
\[
\sum_{n>N}r_n u_n,\qquad |u_n|=1,
\]
is exactly the closed annulus
\[
\{z:\ell_N(\sigma)\le |z|\le R_N(\sigma)\}.
\]

It remains to justify that the tail phases are genuinely free once the finite core phases have been fixed.  Assign the values $\chi(p)=\omega_p$ for $p\in\mathcal P_N$.  Proceed inductively through $n>N$.  Choose a primitive prime $q_n\mid U_n$, write $e_n=v_{q_n}(U_n)$, and assign phase $1$ to any other previously unassigned prime dividing $U_n$.  Since $q_n$ divides no earlier term, its phase is still free; choosing an appropriate $e_n$th root on $\T$ makes $\chi(U_n)$ equal to any prescribed $u_n\in\T$.  Thus every core point and every tail phase sequence can be realized by a completely multiplicative $\chi:\mathbb N\to\T$.

If $Z_U(\sigma+it)=0$, the phases $p^{-it}$ give one such configuration, so the two modulus intervals in \eqref{eq:lucas-core-criterion} must intersect.  Conversely, if they intersect, choose a core phase for which $|C_N|$ equals a common radius and choose tail phases whose sum is $-C_N$.  The preceding construction gives a completely multiplicative $\chi$ with
\[
Z_U(\sigma;\chi):=\sum_{n\ge1}\chi(U_n)U_n^{-\sigma}=0.
\]
The Kronecker approximation argument of Lemma~\ref{lem:vertical}, applied to the finitely many primes occurring in a sufficiently long truncation, gives vertical translates of $Z_U$ converging locally uniformly to $Z_U(\,\cdot\,;\chi)$.  Hurwitz's theorem then gives zeros of $Z_U$ whose real parts tend to $\sigma$.  Continuity of the four functions in \eqref{eq:lucas-core-criterion} gives the converse for limit points.
\end{proof}

\begin{remark}
The value $30$ is universal, not intrinsic.  The explicit classification of defective Lucas numbers in \cite{BHV2001} can reduce the core, and sequence-specific primitive-divisor theorems may reduce it much further.  The Fibonacci calculation in the preceding sections is precisely such a reduction: only the dependencies at $F_6$ and $F_{12}$ survive after the available primitive phases are eliminated.
\end{remark}

The Pell sequence gives a particularly clean illustration.  Let
\[
\mathsf P_0=0,\qquad \mathsf P_1=1,\qquad
\mathsf P_{n+2}=2\mathsf P_{n+1}+\mathsf P_n,
\]
and
\[
Z_{\rm Pell}(s)=\sum_{n\ge1}\mathsf P_n^{-s}.
\]
Carmichael's primitive-divisor theorem implies that every $\mathsf P_n$ with $n>1$ has a primitive prime divisor \cite{Carmichael1913}.  Hence Theorem~\ref{thm:lucas-core} applies with $N=1$: the core is the single constant term $1$.  Since the largest tail length is $2^{-\sigma}<1$, the inner-radius condition is automatic, and the exact criterion is simply
\[
\sum_{n\ge2}\mathsf P_n^{-\sigma}\ge1.
\]

\begin{corollary}[Exact right edge for the Pell zeta function]\label{cor:pell}
There is a unique $\sigma_{\rm Pell}>0$ satisfying
\begin{equation}\label{eq:pell-edge}
Z_{\rm Pell}(\sigma_{\rm Pell})=2.
\end{equation}
Numerically,
\[
\boxed{\sigma_{\rm Pell}=0.875850904477574106444461159947\ldots}.
\]
Moreover,
\[
Z_{\rm Pell}(s)\ne0\qquad(\Ree s\ge\sigma_{\rm Pell}),
\]
and
\[
\overline{\{\Ree\rho:Z_{\rm Pell}(\rho)=0,\ \Ree\rho>0\}}
=[0,\sigma_{\rm Pell}].
\]
No zero lies on the boundary line $\Ree s=\sigma_{\rm Pell}$.
\end{corollary}

\begin{proof}
The tail sum in \eqref{eq:pell-edge} is strictly decreasing from $+\infty$ to $0$, so the solution is unique.  The finite-core criterion gives the zero-free half-plane and sharpness of its boundary in the sense of real-part closure.  If a zero actually lay on the boundary, equality in the triangle inequality would force every tail phase to be $-1$.  In particular
\[
2^{-it}=5^{-it}=-1,
\]
which would make $\log2/\log5$ rational, an impossibility.  Thus the boundary is not attained.  The same Rouch\'e argument as in Theorem~\ref{thm:syndetic} also gives relatively dense ordinates of zeros near every line $0<\sigma\le\sigma_{\rm Pell}$.
\end{proof}

Theorem~\ref{thm:lucas-core} complements the zero-free estimates of Pan and Liu \cite{PanLiu2026}: in the arithmetic Lucas setting it identifies the exact obstruction to sharpness.  Once the finitely many defective terms are isolated, the remaining problem is no longer infinite-dimensional; it is a finite torus optimization coupled to the elementary polygon geometry of the tail.

\section{The Jessen function and mean zero density}\label{sec:jessen}

The preceding results locate zeros topologically and show that they recur with positive lower frequency in every admissible vertical neighborhood.  The classical theory of analytic almost periodic functions turns this into an exact mean-density statement.  For $\sigma>0$ define the Jessen function
\begin{equation}\label{eq:jessen}
\JF(\sigma)
:=\lim_{T\to\infty}\frac1{2T}\int_{-T}^{T}
\log|\ZF(\sigma+it)|\,dt.
\end{equation}
The limit exists and is convex by the theorem of Jessen and Tornehave; see \cite{Jessen1933,JessenTornehave1945} and, for a modern Dirichlet-series formulation, \cite{KouroupisPerfekt2023}.

The phase-realization lemma gives a particularly useful model for this mean.  Let
\[
\mathcal I=\{n\ge3:n\ne6,12\},
\]
and let $(U_n)_{n\in\mathcal I}$ be independent Haar-uniform variables on $\T$.  Put
\begin{align}\label{eq:random-model}
X_\sigma={}&2+2^{-\sigma}U_3+3^{-\sigma}U_4
+8^{-\sigma}U_3^3+144^{-\sigma}U_3^4U_4^2\\
&+\sum_{\substack{n\ge5\\n\ne6,12}}F_n^{-\sigma}U_n.\notag
\end{align}

\begin{proposition}[Haar representation]\label{prop:haar-jessen}
For every $\sigma>0$,
\begin{equation}\label{eq:haar-jessen}
\JF(\sigma)=\mathbb E\log|X_\sigma|.
\end{equation}
If
\[
B_\sigma:=X_\sigma-5^{-\sigma}U_5,
\]
then the independent $U_5$-coordinate can be integrated out exactly:
\begin{equation}\label{eq:conditional-jensen}
\JF(\sigma)
=\mathbb E\log\max\{|B_\sigma|,5^{-\sigma}\}.
\end{equation}
\end{proposition}

\begin{proof}
Lemma~\ref{lem:phase} identifies the character hull of the Fibonacci frequencies with the compact group of phase sequences satisfying only
$u_6=u_3^3$ and $u_{12}=u_3^4u_4^2$.  Under the free coordinates $\mathcal I$, Haar measure is product Haar measure.  The ordinary Dirichlet series $\ZF$ has an analytic spatial extension with respect to the prime-log integral base in the sense of Jessen and Tornehave; their Theorems~29--30 \cite{JessenTornehave1945} therefore identify its Jessen mean with the spatial Haar mean of the logarithmic modulus.

For completeness, the only delicate point in this identification is the logarithmic singularity at zero.  Let
\[
g_\delta(z)=\log\max\{|z|,\delta\},\qquad \delta>0.
\]
Since $X_\sigma$ is a continuous function on the compact character hull, $g_\delta(X_\sigma)$ is continuous and the dense Kronecker translation on the hull is uniquely ergodic.  Hence
\[
\lim_{T\to\infty}\frac1{2T}\int_{-T}^{T}
 g_\delta(\ZF(\sigma+it))\,dt
=\mathbb E\,g_\delta(X_\sigma).
\]
The Jessen--Tornehave spatial-extension theorem is precisely the passage from these continuous regularizations to $\log|\ZF|$.  On the spatial side the required integrability can be seen directly from the free $U_5$ coordinate.  Writing $r=5^{-\sigma}$ and conditioning on all other variables, Jensen's circle formula gives
\[
\frac1{2\pi}\int_0^{2\pi}\log|B_\sigma+re^{i\theta}|\,d\theta
=\log\max\{|B_\sigma|,r\}.
\]
The right-hand side is finite, bounded below by $\log r$, and bounded above uniformly on the compact hull.  Thus the logarithmic singularity is integrable; equivalently, the regularized spatial means decrease to the unregularized one.  This proves \eqref{eq:haar-jessen}, and the same conditional Jensen formula gives \eqref{eq:conditional-jensen}.
\end{proof}

The next regularity fact is useful because it rules out singular concentration of the mean zero measure on isolated vertical lines.

\begin{proposition}[Smooth Jessen function]\label{prop:jessen-smooth}
The function $\JF$ belongs to $C^\infty(0,\infty)$.
\end{proposition}

\begin{proof}
Fix a compact interval $I\Subset(0,\infty)$ and an integer $k\ge0$.  From the free coordinates in \eqref{eq:random-model}, choose $m$ distinct indices $n_j\ge5$ with $n_j\notin\{6,12\}$, with $m$ as large as needed, and split
$X_\sigma=Y_\sigma+W_\sigma$, where
\[
Y_\sigma=\sum_{j=1}^{m}F_{n_j}^{-\sigma}U_{n_j}
\]
is independent of $W_\sigma$.  The restriction $n_j\ge5$ is important: each selected coordinate then occurs only in its linear summand in \eqref{eq:random-model}, whereas $U_3$ and $U_4$ also occur in the forced $F_6$ and $F_{12}$ terms.  The characteristic function of $Y_\sigma$ on $\mathbb R^2\simeq\mathbb C$ is
\[
\widehat\mu_{Y,\sigma}(\xi)
=\prod_{j=1}^{m}J_0(F_{n_j}^{-\sigma}|\xi|),
\]
where $J_0$ is the Bessel function.  The standard bounds for $J_0$ and its derivatives imply, uniformly for $\sigma\in I$,
\[
\partial_\sigma^a\widehat\mu_{Y,\sigma}(\xi)
=O_{I,a,m}\bigl((1+|\xi|)^{a-m/2}\bigr).
\]
The series defining $W_\sigma$ and all its $\sigma$-derivatives converge absolutely and uniformly on $I$, so
\[
\partial_\sigma^b\widehat\mu_{W,\sigma}(\xi)
=O_{I,b}\bigl((1+|\xi|)^b\bigr).
\]
Taking $m>2(k+2)$ shows that the first $k$ $\sigma$-derivatives of the characteristic function of $X_\sigma$ are integrable on $\mathbb R^2$.  Fourier inversion therefore gives a density $p_\sigma(z)$ whose first $k$ $\sigma$-derivatives are bounded and continuous, uniformly on $I$.  Since $X_\sigma$ has uniformly bounded support on $I$ and $\log|z|$ is locally integrable on $\mathbb C$,
\[
\JF(\sigma)=\int_{\mathbb C}\log|z|\,p_\sigma(z)\,dz
\]
may be differentiated $k$ times under the integral.  As $k$ is arbitrary, the claim follows.
\end{proof}

For $T>0$ let
\[
\nu_T:=\frac1{2T}\sum_{\substack{\ZF(\rho)=0\\ |\Im\rho|\le T}}
\delta_{\Ree\rho},
\]
with zeros counted with multiplicity.  Jessen--Tornehave theory identifies the vague limit of these measures with the Riesz measure of $\JF$.

\begin{theorem}[Mean vertical zero density]\label{thm:jessen-density}
Locally on $(0,\infty)$,
\begin{equation}\label{eq:vague-density}
\nu_T\ \Longrightarrow\ d_F(\sigma)\,d\sigma,
\qquad
 d_F(\sigma):=\frac1{2\pi}\JF''(\sigma).
\end{equation}
The density $d_F$ is $C^\infty$ and nonnegative, $\JF$ is strictly convex on $(0,\sigma_F)$, and
\begin{equation}\label{eq:jessen-flat}
\JF(\sigma)=\log2\quad(\sigma\ge\sigma_F),
\qquad
\JF(\sigma)>\log2\quad(0<\sigma<\sigma_F).
\end{equation}
Moreover the support of $d_F(\sigma)\,d\sigma$ is exactly $(0,\sigma_F]$ as a subset of $(0,\infty)$.  Equivalently, every nonempty open interval contained in $(0,\sigma_F]$ carries positive mean zero density.

If
\[
D_F(\sigma):=
\lim_{T\to\infty}\frac1{2T}
\#\{\rho:\ZF(\rho)=0,\ |\Im\rho|\le T,\ \Ree\rho>\sigma\},
\]
then the limit exists for every $\sigma>0$ and
\begin{equation}\label{eq:cumulative-density}
D_F(\sigma)=-\frac1{2\pi}\JF'(\sigma).
\end{equation}
In particular $D_F$ is $C^\infty$ and strictly decreasing on $(0,\sigma_F)$, with $D_F(\sigma)>0$ there, while $D_F(\sigma)=0$ for $\sigma\ge\sigma_F$.
Finally the cutoff is flat to every algebraic order: for each $M\ge1$,
\begin{equation}\label{eq:flat-edge-density}
\JF(\sigma)-\log2
=O_M((\sigma_F-\sigma)^M),
\qquad
D_F(\sigma)=O_M((\sigma_F-\sigma)^M)
\end{equation}
as $\sigma\uparrow\sigma_F$.
\end{theorem}

\begin{proof}
The convergence \eqref{eq:vague-density} is the classical Jessen--Tornehave zero-density theorem; in the notation of \cite{JessenTornehave1945}, this is the zero-frequency conclusion of Theorem~29 for the analytic spatial extension furnished by Theorem~30.  Proposition~\ref{prop:jessen-smooth} makes its Riesz measure absolutely continuous with the displayed smooth density.

By Proposition~\ref{prop:zerofree}, there are no zeros in $\Ree s\ge\sigma_F$.  Hence $\JF''=0$ on $(\sigma_F,\infty)$, so $\JF$ is affine there.  Since $\ZF(\sigma+it)\to2$ uniformly in $t$ as $\sigma\to\infty$, one has $\JF(\sigma)\to\log2$; therefore $\JF\equiv\log2$ on $[\sigma_F,\infty)$.

Let $\mu=d_F(\sigma)\,d\sigma$ denote the vague limit in \eqref{eq:vague-density}.  To justify the support statement without using Portmanteau in the wrong direction, let $I\subset(0,\infty)$ be any nonempty open interval meeting $(0,\sigma_F]$.  Because $I$ is open, we may choose $\sigma_0\in I\cap(0,\sigma_F)$ and a compact interval
\[
K=[\sigma_0-r,\sigma_0+r]\Subset I
\]
with $r>0$ small.  Apply Theorem~\ref{thm:syndetic} with horizontal tolerance smaller than $r$.  Its linear counting conclusion gives a constant $c_K>0$ such that
\[
\nu_T(K)\ge c_K
\]
for all sufficiently large $T$.  Since $K$ is compact, the Portmanteau theorem for vague convergence gives
\[
\limsup_{T\to\infty}\nu_T(K)\le\mu(K),
\]
and hence $\mu(K)\ge c_K>0$.  Therefore $\mu(I)>0$.  This proves that the support is exactly $(0,\sigma_F]$ relative to $(0,\infty)$.

The Riesz measure consequently gives positive mass to every nonempty subinterval of $(0,\sigma_F)$, so $\JF$ is strictly convex there.  Convexity and $\JF'=0$ on $(\sigma_F,\infty)$ then give $\JF'(\sigma)<0$ for $\sigma<\sigma_F$, proving the strict inequality in \eqref{eq:jessen-flat}.  The same positive-mass statement shows that $D_F=-\JF'/(2\pi)$ is strictly decreasing on $(0,\sigma_F)$.

Integrating \eqref{eq:vague-density} over $(\sigma,\infty)$ and using smoothness (hence absence of atoms on the boundary) gives
\[
D_F(\sigma)=\int_\sigma^\infty d_F(u)\,du
=-\frac1{2\pi}\JF'(\sigma),
\]
which is \eqref{eq:cumulative-density}.  Finally, $\JF-\log2$ is $C^\infty$ and vanishes identically to the right of $\sigma_F$; all of its derivatives therefore vanish at $\sigma_F$.  Taylor's theorem gives \eqref{eq:flat-edge-density} for every $M$.
\end{proof}

\begin{remark}\label{rem:conditional-derivative}
Formula \eqref{eq:conditional-jensen} also gives a stable numerical representation of the cumulative density.  Since $B_\sigma$ has a smooth planar density, the event $|B_\sigma|=5^{-\sigma}$ has probability zero, and differentiation yields
\[
D_F(\sigma)=\frac1{2\pi}\mathbb E\!\left[
(\log5)\mathbf 1_{\{|B_\sigma|<5^{-\sigma}\}}
-\mathbf 1_{\{|B_\sigma|>5^{-\sigma}\}}
\Ree\frac{B_\sigma'}{B_\sigma}
\right].
\]
This avoids direct zero searches at very large height and is the basis of the numerical values reported below.

\end{remark}

\section{Approach to the right edge: phase locking and height}\label{sec:height}

The existence and density results above do not by themselves say how high one must go before seeing a zero within a prescribed distance of the right edge.  This section gives a quantitative necessary condition for such zeros and separates three notions: the first observed height $T_F(\varepsilon)$, the mean spacing supplied by the Jessen density, and a shrinking-target heuristic for the former.

For $\sigma\ge1/2$ retain the notation $A,B,C,D,L,R,H$ from Section~2 and put
\[
K(\sigma):=72^{-\sigma}+18^{-\sigma}+4\,48^{-\sigma}<1.
\]

\begin{proposition}[Quantitative phase locking]\label{prop:phase-lock}
Let $\rho=\sigma+it$ be a zero of $\ZF$ with $1/2\le\sigma<\sigma_F$.  Then
\begin{equation}\label{eq:lock23}
A|2^{-it}+1|^2+B|3^{-it}+1|^2
\le \frac{2H(\sigma)}{1-K(\sigma)}.
\end{equation}
Write
\[
E(\rho)=\sum_{\substack{n\ge5\\n\ne6,12}}F_n^{-\sigma}F_n^{-it}
\]
and $\omega=E(\rho)/|E(\rho)|$.  Then
\begin{equation}\label{eq:lock-free-energy}
\sum_{\substack{n\ge5\\n\ne6,12}}
F_n^{-\sigma}|F_n^{-it}-\omega|^2\le2H(\sigma),
\end{equation}
and
\begin{equation}\label{eq:lock-direction}
|\omega+1|^2\le\frac{2H(\sigma)}{L(\sigma)+H(\sigma)}.
\end{equation}
Consequently, for every $n\ge5$, $n\ne6,12$,
\begin{equation}\label{eq:lock-individual}
|F_n^{-it}+1|
\le \sqrt{2H(\sigma)F_n^{\sigma}}
+\sqrt{\frac{2H(\sigma)}{L(\sigma)+H(\sigma)}}.
\end{equation}
\end{proposition}

\begin{proof}
Use the notation in the proof of Lemma~\ref{lem:torus}, so that
$z=-e^{iu}$, $w=-e^{iv}$ and
\[
P=2+Az+Bw+Cz^3+Dz^4w^2.
\]
The weighted Cauchy--Schwarz estimate used there actually gives
\begin{align*}
\Ree P-L(\sigma)
&\ge \frac{1-K(\sigma)}2
\bigl(A|e^{iu}-1|^2+C|e^{3iu}-1|^2+B|e^{iv}-1|^2\bigr).
\end{align*}
At a zero, $E(\rho)=-P$ and
\[
\Ree P\le |P|=|E(\rho)|\le R(\sigma)=L(\sigma)+H(\sigma).
\]
Dropping the nonnegative $C$-term gives \eqref{eq:lock23}.

Next set $a_n=F_n^{-\sigma}$ and $u_n=F_n^{-it}$ on the free indices.  Since
$\sum a_n=R(\sigma)$ and $|E(\rho)|=|P|\ge L(\sigma)$,
\[
R(\sigma)-|E(\rho)|\le H(\sigma).
\]
Therefore
\[
\sum a_n|u_n-\omega|^2
=2\bigl(R(\sigma)-|E(\rho)|\bigr)\le2H(\sigma),
\]
which is \eqref{eq:lock-free-energy}.  Finally $\omega=-P/|P|$, and hence
\[
|\omega+1|^2
=2\left(1-\frac{\Ree P}{|P|}\right)
\le2\left(1-\frac{L(\sigma)}{L(\sigma)+H(\sigma)}\right),
\]
proving \eqref{eq:lock-direction}.  The triangle inequality together with the $n$th summand of \eqref{eq:lock-free-energy} gives \eqref{eq:lock-individual}.
\end{proof}

The proposition shows that the extremal phase configuration is not merely the configuration that proves the right edge: every sequence of actual zeros approaching that edge is forced toward it.

\begin{corollary}[Growing-dimensional phase lock]\label{cor:growing-lock}
Let $\rho_j=\sigma_j+it_j$ be zeros with $\sigma_j\uparrow\sigma_F$, and put $\varepsilon_j=\sigma_F-\sigma_j$.  Then
\[
2^{-it_j},\ 3^{-it_j},\ 8^{-it_j}\longrightarrow-1,
\qquad
144^{-it_j}\longrightarrow+1,
\]
and $F_n^{-it_j}\to-1$ for every fixed $n\ge3$, $n\ne12$.
More quantitatively, if
\[
0<c<\frac1{\sigma_F\log\varphi}=2.7962718895\ldots,
\]
then
\begin{equation}\label{eq:growing-lock}
\max_{\substack{3\le n\le c\log(1/\varepsilon_j)\\n\ne12}}
|F_n^{-it_j}+1|
=O_c\!\left(\varepsilon_j^{(1-c\sigma_F\log\varphi)/2}\right).
\end{equation}
\end{corollary}

\begin{proof}
Since $H(\sigma_F)=0$ and $H'(\sigma_F)<0$,
\[
H(\sigma_F-\varepsilon)=-H'(\sigma_F)\varepsilon+O(\varepsilon^2).
\]
Equations \eqref{eq:lock23}--\eqref{eq:lock-individual} therefore give the fixed-index limits; $F_6=2^3$ and $F_{12}=2^4 3^2$ give the two dependent phases.  For the uniform statement use $F_n\le\varphi^{n-1}$ in \eqref{eq:lock-individual}; the $2$- and $3$-terms obey the stronger estimate \eqref{eq:lock23}.  The exponent in \eqref{eq:growing-lock} is positive precisely for the displayed range of $c$.
\end{proof}

This phase locking can be converted into a rigorous, if not expected to be sharp, lower bound for the height needed to approach the edge.

\begin{proposition}[A Diophantine zero-free cusp]\label{prop:cusp}
For every $\eta>0$,
\begin{equation}\label{eq:T-lower}
T_F(\varepsilon)\gg_\eta
\varepsilon^{-1/(15.232+\eta)}
\qquad(\varepsilon\downarrow0).
\end{equation}
Equivalently, after possibly changing the implied constant, every zero $\rho=\beta+i\gamma$ sufficiently close to the right edge satisfies
\begin{equation}\label{eq:cusp}
\sigma_F-\beta\gg_\eta(1+|\gamma|)^{-15.232-\eta}.
\end{equation}
\end{proposition}

\begin{proof}
Put $\varepsilon=\sigma_F-\sigma$.  By \eqref{eq:lock23} and the preceding Taylor expansion,
\[
|2^{-it}+1|+|3^{-it}+1|\ll\sqrt\varepsilon.
\]
For $\varepsilon$ small there are therefore odd integers $a,b$ such that
\[
|t\log2-a\pi|+|t\log3-b\pi|\ll\sqrt\varepsilon.
\]
Consequently
\begin{equation}\label{eq:linear-form-small}
|a\log3-b\log2|\ll\sqrt\varepsilon,
\end{equation}
while $\max(|a|,|b|)\asymp1+|t|$.

Rhin's estimate may be stated as the effective irrationality-measure bound
\[
\left|\frac{\log3}{\log2}-\frac{p}{q}\right|
\gg_\delta q^{-8.616-\delta}
\qquad(q\to\infty)
\]
for every $\delta>0$; see \cite[Eq.~(8)]{Rhin1987}.  Taking $q=|a|$ and the corresponding signed numerator $p$ from $b/a$, and multiplying by $|a|\log2$, gives
\[
|a\log3-b\log2|\gg_\delta
\max(|a|,|b|)^{-7.616-\delta},
\]
because $|a|\asymp|b|$ in the present situation.  Thus the loss from $8.616$ to $7.616$ is exactly the factor $|a|$ introduced when passing from rational approximation to the linear form.  Combining this with \eqref{eq:linear-form-small} gives
\[
1+|t|\gg_\delta\varepsilon^{-1/(15.232+2\delta)}.
\]
Renaming $2\delta$ as $\eta$ proves both assertions.
\end{proof}

\begin{remark}[Mean spacing versus first hitting time]\label{rem:mean-v-first}
The Jessen result yields a much stronger statement on the \emph{mean} scale, but it should not be confused with a bound for the first zero.  Define
\[
M_F(\varepsilon):=\frac1{D_F(\sigma_F-\varepsilon)}.
\]
The infinite-order flatness in \eqref{eq:flat-edge-density} gives, for every $M\ge1$,
\begin{equation}\label{eq:mean-superpoly}
M_F(\varepsilon)\gg_M\varepsilon^{-M}.
\end{equation}
Thus the reciprocal mean density of zeros in the $\varepsilon$-edge strip grows faster than every power of $1/\varepsilon$.  This does not by itself imply the same lower bound for $T_F(\varepsilon)$, because a very early exceptional hit is compatible with a much larger asymptotic mean spacing.
\end{remark}

There is nevertheless a natural shrinking-target model for the first-hitting problem.  Put
\[
\lambda_F:=\sigma_F\log\varphi
=0.357619015418870\ldots.
\]
At the edge the free weights satisfy
$F_n^{-\sigma_F}\asymp e^{-\lambda_F n}$.  Proposition~\ref{prop:phase-lock} says, at the level of necessary conditions, that a near-edge zero with gap $\varepsilon$ must lie in a weighted phase ball of energy $O(\varepsilon)$.  The relevant coordinates are those with $F_n^{-\sigma_F}\gtrsim\varepsilon$, hence their number is
\[
m\sim\frac{\log(1/\varepsilon)}{\lambda_F}.
\]
The angular width in the $n$th such free coordinate is of order
$\sqrt{\varepsilon/F_n^{-\sigma_F}}$.  Multiplying these widths, and including the $m$-dimensional ball factor, gives the logarithmic Haar-volume estimate
\begin{equation}\label{eq:target-volume}
\log V(\varepsilon)
=-\frac{(\log(1/\varepsilon))^2}{4\lambda_F}
+O\bigl(\log(1/\varepsilon)\log\log(1/\varepsilon)\bigr).
\end{equation}
This leads to the following deliberately heuristic prediction.

\begin{remark}[Shrinking-target prediction]\label{rem:height-heuristic}
If the growing-dimensional Kronecker flow generated by the prime logarithms hits the near-extremal phase target on the inverse-Haar-volume scale suggested by \eqref{eq:target-volume}, then
\begin{equation}\label{eq:T-heuristic}
\log T_F(\varepsilon)
\sim\frac{(\log(1/\varepsilon))^2}{4\sigma_F\log\varphi}
=0.6990679724\ldots\,(\log(1/\varepsilon))^2.
\end{equation}
Equivalently $T_F(\varepsilon)$ would be superpolynomial but subexponential in $1/\varepsilon$.  We do not prove \eqref{eq:T-heuristic}: first-hitting times of a deterministic Kronecker flow can fluctuate with its Diophantine recurrence.  Quantitative local Kronecker estimates for logarithms of primes, such as \cite{KorolevRezvyakova2022}, provide a possible route to rigorous upper bounds, but the dimension here grows like $\log(1/\varepsilon)$.
\end{remark}

\paragraph{Exploratory near-edge zeros.}
A phase-guided Newton search, followed by $70$-digit evaluation of the full convergent series, produced the following representative zeros.  The search was targeted rather than exhaustive, so the table is evidence about the scale of the problem, not a certification of the exact function $T_F(\varepsilon)$.
\[
\begin{array}{r@{.}l r@{.}l c}
\multicolumn{2}{c}{\Ree\rho}&\multicolumn{2}{c}{|\Im\rho|}&\sigma_F-\Ree\rho\\
\hline
0&678781122233&3222&603994&0.06438228\\
0&687329021982&13900&499214&0.05583438\\
0&694991090210&44675&348873&0.04817231\\
0&704717283154&197035&095820&0.03844612\\
0&716227677984&260134&869444&0.02693572\\
0&719828726132&6874380&278267&0.02333467\\
0&720764507438&7026740&039690&0.02239889\\
0&730178633163&8596695&206382&0.01298477
\end{array}
\]
Within the same targeted search up to height $10^8$, no root with larger real part than the last entry was found.  This negative search statement is not used in any proof.

\section{Exact right edges for the partial sums}

For $N\ge12$ write
\[
\ZFN(s):=\sum_{n=1}^{N}F_n^{-s}
\]
and define
\begin{equation}\label{eq:HN}
H_N(\sigma):=\ZFN(\sigma)-4-2\,144^{-\sigma}.
\end{equation}
The same four-term obstruction determines the exact positive right edge of every such Dirichlet polynomial.

\begin{theorem}\label{thm:partial}
For every integer $N\ge12$ there is a unique $\sigma_N\in(1/2,3/4)$ satisfying
\begin{equation}\label{eq:sigmaN}
\ZFN(\sigma_N)=4+2\,144^{-\sigma_N}.
\end{equation}
If
\[
\mathcal R_{F,N}^{+}
:=\{\Ree\rho:\ZFN(\rho)=0,\ \Ree\rho>0\},
\]
then
\[
\overline{\mathcal R_{F,N}^{+}}=[0,\sigma_N].
\]
In particular,
\[
\sup\mathcal R_{F,N}^{+}=\sigma_N,
\]
and no zero of $\ZFN$ lies on the line $\Ree s=\sigma_N$.  Moreover
\[
\sigma_N\nearrow\sigma_F.
\]
\end{theorem}

\begin{proof}
For $N\ge12$, differentiation of \eqref{eq:HN} gives
\[
H_N'(\sigma)
=-\!\sum_{\substack{3\le n\le N\\ n\ne12}}
(\log F_n)F_n^{-\sigma}
+(\log144)144^{-\sigma}.
\]
As in Lemma~\ref{lem:H}, the terms $n=3,4,5,6$ already dominate the positive term for every $\sigma\ge0$, so $H_N'(\sigma)<0$.  A direct finite calculation gives
\[
H_{12}(1/2)=0.9097970371\ldots>0,
\]
and hence $H_N(1/2)>0$.  On the other hand
\[
H_N(3/4)<H(3/4)<0.
\]
Thus \eqref{eq:sigmaN} has a unique solution in $(1/2,3/4)$.

Put
\[
R_N(\sigma)=\sum_{\substack{5\le n\le N\\n\ne6,12}}F_n^{-\sigma}.
\]
Then $H_N(\sigma)=R_N(\sigma)-L(\sigma)$.  If $\ZFN(\sigma+it)=0$ with $\sigma\ge1/2$, the decomposition used in Proposition~\ref{prop:zerofree}, together with Lemma~\ref{lem:torus}, gives
\[
L(\sigma)\le R_N(\sigma),
\]
hence $H_N(\sigma)\ge0$ and therefore $\sigma\le\sigma_N$.  Equality $\sigma=\sigma_N$ would force equality in Lemma~\ref{lem:torus}, hence $2^{-it}=3^{-it}=-1$, which is impossible.  Thus the boundary is not attained.

To prove sharpness and the density of the positive real projections, fix $0<\sigma\le\sigma_N$.  The free lengths have total mass $R_N(\sigma)$ and largest member $5^{-\sigma}$.  Since $\sigma_N<3/4$ and $N\ge12$,
\begin{align*}
\frac{R_N(\sigma)}{5^{-\sigma}}
&\ge 1+\left(\frac5{13}\right)^{3/4}
       +\left(\frac5{21}\right)^{3/4}
       +\left(\frac5{34}\right)^{3/4}
       +\left(\frac5{55}\right)^{3/4}
       +\left(\frac5{89}\right)^{3/4}>2.
\end{align*}
Also $|L(\sigma)|\le R_N(\sigma)$, by the same argument as in Lemma~\ref{lem:twistzero}, since $H_N(\sigma)\ge0$.  The polygon lemma therefore supplies phases of the free terms whose sum is $-L(\sigma)$, and Lemma~\ref{lem:phase} realizes them by a completely multiplicative twist.  Kronecker approximation and Hurwitz's theorem then give zeros of $\ZFN$ with real parts arbitrarily close to $\sigma$.  Hence
\[
\overline{\mathcal R_{F,N}^{+}}=[0,\sigma_N].
\]

Finally, $H_{N+1}(\sigma)>H_N(\sigma)$ for every $\sigma>0$, so the roots $\sigma_N$ increase.  Since $H_N(\sigma_F)<H(\sigma_F)=0$, they are bounded above by $\sigma_F$.  If their limit were $\lambda<\sigma_F$, choose $\sigma\in(\lambda,\sigma_F)$.  Then $H_N(\sigma)<0$ for every sufficiently large $N$, whereas $H_N(\sigma)\to H(\sigma)>0$, a contradiction.  Thus $\sigma_N\to\sigma_F$.
\end{proof}

The convergence of these exact edges is exponentially fast.

\begin{corollary}\label{cor:rate}
As $N\to\infty$,
\begin{equation}\label{eq:rate}
\sigma_F-\sigma_N
\sim
C_F\,\varphi^{-(N+1)\sigma_F},
\end{equation}
where
\[
C_F=
\frac{5^{\sigma_F/2}}
{(-H'(\sigma_F))(1-\varphi^{-\sigma_F})}
=1.77646645548344\ldots.
\]
\end{corollary}

\begin{proof}
Let
\[
T_N(\sigma)=\sum_{n>N}F_n^{-\sigma},
\]
so that $H_N=H-T_N$.  Since $H(\sigma_F)=0$ and $H'(\sigma_F)<0$, the mean-value theorem gives
\[
\sigma_F-\sigma_N
=\frac{T_N(\sigma_N)}{-H'(\xi_N)}
\]
for some $\xi_N\in(\sigma_N,\sigma_F)$.  Theorem~\ref{thm:partial} gives $\sigma_N\to\sigma_F$, so $H'(\xi_N)\to H'(\sigma_F)$.  Since $-H'$ is bounded away from zero near $\sigma_F$ and $F_n\gg\varphi^n$, the displayed identity also gives
\[
\sigma_F-\sigma_N=O(\varphi^{-N\sigma_F/2}),
\]
and hence $N(\sigma_F-\sigma_N)\to0$.  The geometric tail asymptotic from Binet's formula, uniformly for $\sigma$ near $\sigma_F$, then shows
\[
\frac{T_N(\sigma_N)}{T_N(\sigma_F)}\longrightarrow1.
\]
Consequently
\[
\sigma_F-\sigma_N
\sim \frac{T_N(\sigma_F)}{-H'(\sigma_F)}.
\]

Binet's formula gives
\[
F_n^{-\sigma_F}
=5^{\sigma_F/2}\varphi^{-n\sigma_F}
\bigl(1+O(\varphi^{-2n})\bigr),
\]
and therefore
\[
T_N(\sigma_F)
\sim
\frac{5^{\sigma_F/2}\varphi^{-(N+1)\sigma_F}}
{1-\varphi^{-\sigma_F}}.
\]
Substitution yields \eqref{eq:rate}.
\end{proof}

\section{A global pole-removing completion}\label{sec:completion}

The preceding sections deliberately used only the half-plane $\Ree s>0$.  The known meromorphic continuation has a very different geometry on the left: Egami and Navas proved, and the recent treatment in \cite{AssafKuanLowryDudaWalker2024} recalls, that
\begin{equation}\label{eq:global-continuation}
\ZF(s)
=5^{s/2}\sum_{k\ge0}\binom{-s}{k}
\frac{1}{\varphi^{s+2k}+(-1)^{k+1}},
\end{equation}
with simple poles
\begin{equation}\label{eq:pole-lattice}
p_{k,m}
=-2k+\frac{(2m+k)\pi i}{\log\varphi},
\qquad k\ge0,\quad m\in\mathbb Z.
\end{equation}
Navas also identified the real trivial zeros $-2,-6,-10,\ldots$; see \cite{Navas2001}.  We do not attempt to classify all nonreal zeros in $\Ree s<0$.  Instead, the pole lattice can be removed exactly, which gives a useful global framework for the zero divisor.

Put
\[
\Lambda=\log\varphi,
\qquad
q=-\varphi^{-2},
\]
and, for $|q|<1$, use the standard notation
\[
(a;q)_\infty=\prod_{j=0}^{\infty}(1-aq^j),
\qquad
(q;q)_k=\prod_{j=1}^{k}(1-q^j).
\]
Define
\begin{equation}\label{eq:completion-def}
P_F(s):=(\varphi^{-s};q)_\infty,
\qquad
\Xi_F(s):=P_F(s)\ZF(s).
\end{equation}

\begin{theorem}[Global completion]\label{thm:completion}
The function $\Xi_F$ extends to an entire function.  Its zeros, counted with multiplicity, are exactly the zeros of the meromorphic function $\ZF$: none of the poles in \eqref{eq:pole-lattice} becomes a zero after cancellation.  More precisely,
\begin{equation}\label{eq:completion-pole-value}
\Xi_F(p_{k,m})
=q^{-k(k+1)/2}(q;q)_k(q;q)_\infty
\,5^{p_{k,m}/2}\binom{-p_{k,m}}{k}\neq0.
\end{equation}
If
\[
M_\Xi(r)=\max_{|s|\le r}|\Xi_F(s)|,
\]
then
\begin{equation}\label{eq:completion-type}
\limsup_{r\to\infty}\frac{\log M_\Xi(r)}{r^2}
=\frac{\log\varphi}{4}.
\end{equation}
Thus $\Xi_F$ has exact order $2$ and exact type $\log\varphi/4$.

Let $n_\infty(R)$ be the number of poles of $\ZF$ in $|s|\le R$, and let $n_0(R)$ be the number of zeros there, both counted with multiplicity.  Then
\begin{equation}\label{eq:global-counting}
n_\infty(R)
=\frac{\log\varphi}{8}R^2+O(R),
\qquad
n_0(R)=O(R^2).
\end{equation}
\end{theorem}

\begin{proof}
The $k$th denominator in \eqref{eq:global-continuation} factors as
\[
\varphi^{s+2k}+(-1)^{k+1}
=\varphi^{s+2k}\bigl(1-\varphi^{-s}q^k\bigr).
\]
Consequently the zeros of $P_F$ are precisely the points \eqref{eq:pole-lattice}, and they are simple.  At $p=p_{k,m}$ one has $\varphi^{-p}q^k=1$, and therefore
\begin{align*}
P_F'(p)
&=\Lambda\prod_{j\ne k}(1-q^{j-k})\\
&=\Lambda(-1)^kq^{-k(k+1)/2}(q;q)_k(q;q)_\infty.
\end{align*}
On the other hand, the residue of \eqref{eq:global-continuation} at the same point is
\[
\operatorname*{Res}_{s=p}\ZF(s)
=\frac{5^{p/2}\binom{-p}{k}}
{\Lambda(-1)^k}.
\]
Multiplying these two expressions proves \eqref{eq:completion-pole-value}.  The right-hand side is nonzero: $(q;q)_\infty\ne0$ since $|q|<1$, and $\Ree(-p_{k,m})=2k$ prevents $\binom{-p_{k,m}}{k}$ from vanishing.  Hence every singularity is removable in \eqref{eq:completion-def}, and the cancellation introduces no new zero.

For the growth estimate, write
\[
P_{F,k}(s)=\prod_{\substack{j\ge0\\j\ne k}}
(1-\varphi^{-s}q^j).
\]
For every fixed $k$, the product $P_{F,k}$ converges locally uniformly and defines an entire function.  Moreover, for $|s|\le r$, $r\ge1$,
\[
|P_{F,k}(s)|
\le Q(r):=\prod_{j\ge0}(1+\varphi^{r-2j}).
\]
With $x=\varphi^{-2}$ one also has
\[
\sum_{k\ge0}\left|\binom{-s}{k}\right|x^k
\le (1-x)^{-r}.
\]
Consequently
\[
\sum_{k\ge0}\left|\binom{-s}{k}\right|\varphi^{-2k}|P_{F,k}(s)|
\le Q(r)(1-\varphi^{-2})^{-r},
\]
so the series below converges absolutely and locally uniformly on $\mathbb C$ by the Weierstrass test.  Off the pole lattice direct cancellation in \eqref{eq:global-continuation} gives
\begin{equation}\label{eq:Xi-series}
\Xi_F(s)
=5^{s/2}\varphi^{-s}
\sum_{k\ge0}\binom{-s}{k}\varphi^{-2k}P_{F,k}(s),
\end{equation}
and since the right-hand side is entire, the identity holds everywhere by continuation.

Splitting the logarithm of $Q(r)$ at $j=r/2$ gives
\begin{equation}\label{eq:qproduct-growth}
\sum_{j\ge0}\log(1+\varphi^{r-2j})
\le \frac{\Lambda}{4}r^2+O(r).
\end{equation}
The remaining prefactor in \eqref{eq:Xi-series}, as well as the factor $(1-\varphi^{-2})^{-r}$ already displayed above, is only exponential in $r$.  Thus
\[
\log M_\Xi(r)\le \frac{\Lambda}{4}r^2+O(r).
\]

For the reverse bound, use the real pole subsequence $p_{2N,-N}=-4N$.  Formula \eqref{eq:completion-pole-value} yields
\[
|\Xi_F(-4N)|
=|q|^{-N(2N+1)}
|(q;q)_{2N}(q;q)_\infty|
\,5^{-2N}\binom{4N}{2N}.
\]
Since $(q;q)_{2N}\to(q;q)_\infty\ne0$, Stirling's formula gives
\[
\log|\Xi_F(-4N)|
=4\Lambda N^2+O(N)
=\frac{\Lambda}{4}(4N)^2+O(N),
\]
which proves \eqref{eq:completion-type}.

Finally, for fixed $k$ the poles in \eqref{eq:pole-lattice} have vertical spacing $2\pi/\Lambda$.  Hence
\begin{align*}
n_\infty(R)
&=\sum_{0\le k\le R/2}
\left(\frac{\Lambda}{\pi}\sqrt{R^2-4k^2}+O(1)\right)\\
&=\frac{\Lambda}{\pi}
\int_0^{R/2}\sqrt{R^2-4x^2}\,dx+O(R)\\
&=\frac{\Lambda}{8}R^2+O(R).
\end{align*}
The bound $n_0(R)=O(R^2)$ follows from Jensen's formula applied to the entire function $\Xi_F$ and the order-two bound above.
\end{proof}

The theorem packages the global zero/pole geometry into a single entire function.  Since $n_0(R)=O(R^2)$, partial summation gives
\[
\sum_{\rho}|\rho|^{-3}<\infty,
\]
where the sum is over the zeros with multiplicity.  Thus the genus-$2$ canonical product converges.  Writing
\[
E_2(z)=(1-z)\exp\!\left(z+\frac{z^2}{2}\right),
\]
Hadamard factorization for an entire function of order $2$ gives constants $A,B,C\in\C$ such that
\begin{equation}\label{eq:hadamard-completion}
\Xi_F(s)
=e^{As^2+Bs+C}
\prod_{\rho}E_2\!\left(\frac{s}{\rho}\right)
=e^{As^2+Bs+C}
\prod_{\rho}
\left(1-\frac{s}{\rho}\right)
\exp\!\left(\frac{s}{\rho}+\frac{s^2}{2\rho^2}\right),
\end{equation}
where $\rho$ runs over all zeros of $\ZF$, with multiplicity.  Thus the explicitly known pole half-lattice is separated cleanly from the still much subtler global zero divisor.  The trivial zeros $\rho=-4j-2$ found by Navas occur among the factors in \eqref{eq:hadamard-completion}.  No analogue of a Riemann-type functional equation is known; compare \cite{Navas2001}.  Thus \eqref{eq:hadamard-completion} should be viewed as a global bookkeeping device rather than evidence for a reflected ``critical strip.''

\section{Numerical value and comparison}

The defining equation for the sharp edge is
\[
\ZF(\sigma_F)=4+2\,144^{-\sigma_F},
\]
which gives
\[
\boxed{\sigma_F=0.743163398726901648018218067138\ldots}.
\]
For comparison, the earlier triangle-inequality boundary $\eta$ is defined by $\ZF(\eta)=4$ and equals
\[
\eta=0.7570549496906548985355124\ldots.
\]
Thus the exceptional relation at $F_{12}=144$ shifts the exact right edge left by about
\[
\eta-\sigma_F=0.01389155096375325\ldots.
\]

For the partial sums, the first few exact right edges are
\[
\begin{array}{c|ccccc}
N&12&15&20&30&\infty\\
\hline
\sigma_N&0.7243955687&0.7371070760&0.7421818981&0.7431361685&0.7431633987
\end{array}
\]
and Corollary~\ref{cor:rate} explains the rapid stabilization.

\medskip
\noindent\textbf{Exploratory Jessen-density values.}
Using the exact conditional formula \eqref{eq:conditional-jensen}, a scrambled Sobol quasi-Monte Carlo calculation gives the following indicative values.  Here $D_F(\sigma)$ is the mean number of zeros to the right of $\sigma$, per unit vertical length in the normalization \eqref{eq:cumulative-density}.
\[
\begin{array}{c|cc}
\sigma&\JF(\sigma)&D_F(\sigma)\\
\hline
0.10&1.0327&0.5670\\
0.20&0.82270&0.1912\\
0.30&0.74365&0.07859\\
0.40&0.71082&0.03241\\
0.50&0.69763&0.01212\\
0.60&0.693569&0.00220\\
0.65&0.693199&0.00044\\
\sigma_F&\log2&0
\end{array}
\]
The rapid collapse of $D_F$ near $\sigma_F$ is consistent with the rigorous infinite-order flatness in \eqref{eq:flat-edge-density}; in particular one should not expect a simple power-law onset at the right edge.

The main point is structural rather than numerical: the rightmost possible phase configuration is not obtained by assigning independent signs to all Fibonacci terms.  The relations $F_6=2^3$ and $F_{12}=2^4 3^2$ force the $F_{12}$ term to point in the opposite direction from the terms $F_3,F_4,F_6$ at the extremal twist, producing the correction $2\,144^{-\sigma}$ in \eqref{eq:sigma-def}.  Carmichael's theorem then shows that, beyond these exceptional dependencies, all remaining Fibonacci terms have enough new prime phase freedom to make the bound sharp in the sense of Theorem~\ref{thm:main}.

\section*{Reproducibility statement}
A frozen reproducibility snapshot corresponding to arXiv v1, including all scripts described below, their recorded outputs, the LaTeX source, and the compiled manuscript, is archived on Zenodo (DOI: \href{https://doi.org/\ZenodoDOI}{\texttt{\ZenodoDOI}}).  The archive also contains a README file, a minimal Python requirements file, citation metadata, a file manifest, and SHA-256 checksums.  The script \texttt{verify.py} evaluates $\sigma_F$, $\eta$, the sign checks in Lemma~\ref{lem:H}, the partial-sum edges $\sigma_N$, and the asymptotic constant $C_F$.  The script \texttt{jessen\_density.py} implements the independent-phase model and the conditional formula \eqref{eq:conditional-jensen} using scrambled Sobol sampling to reproduce the exploratory Jessen-density table.  The script \texttt{edge\_height.py} refines and verifies the near-edge zeros in Section~\ref{sec:height} at high precision and reports the shrinking-target normalization.  The exploratory phase-guided scan is implemented in \texttt{edge\_height\_search.py}; the recorded $10^8$ run is supplied as \texttt{edge\_height\_search\_100m.txt}.  The script \texttt{lucas\_core\_check.py} evaluates the Pell edge in Corollary~\ref{cor:pell}.  Finally, \texttt{completion\_check.py} numerically checks the removable-value formula \eqref{eq:completion-pole-value} and the approach to the exact type in \eqref{eq:completion-type}.  None of the rigorous results depends on these numerical experiments.

\section*{AI-Disclosure}
The author used OpenAI's ChatGPT 5.6 Sol for help in language editing, LaTeX restructuring, and the preparation of verification scripts.  All mathematical statements, proofs, computations, references, and the final presentation were independently verified by the author, who assumes full responsibility for the content.

\end{document}